\documentclass[10pt]{amsart}
\usepackage{amscd}
\usepackage{amsmath}
\usepackage{amsfonts}
\usepackage{amssymb}
\usepackage{amsthm}
\usepackage{enumerate}
\usepackage[T1]{fontenc}
\usepackage{hyperref}
\usepackage{mathrsfs}
\usepackage{stackrel}
\usepackage{verbatim}
\usepackage[all]{xy}
\usepackage{yhmath}

\newtheorem{thm}{Theorem}[section]
\newtheorem{lem}[thm]{Lemma}
\newtheorem{defn}[thm]{Definition}
\newtheorem{prop}[thm]{Proposition}
\newtheorem{cor}[thm]{Corollary}

\newcommand{\B}{\mathcal{B}}
\newcommand{\C}{\mathcal{C}}
\newcommand{\D}{\mathcal{D}}

\newcommand{\M}{\mathcal{M}}

\renewcommand{\O}{\mathcal{O}}

\newcommand{\h}[1]{\widehat{#1}}

\newcommand{\w}[1]{\wideparen{#1}}

\newcommand{\cind}{\operatorname{c-Ind}}

\begin{document}
	\title{A Mackey criterion for locally analytic representations}
	\author{Andreas Bode}
	\maketitle

	\begin{abstract}
		We prove an analogue of Mackey's irreducibility criterion for compactly induced locally analytic representations of $p$-adic groups, where we induce from an open subgroup which is compact mod centre. We also discuss several examples.
	\end{abstract}
	
	\tableofcontents
	
	\section{Introduction}
	
	This paper studies the locally analytic induction with compact support of an admissible, locally analytic representation of a $p$-adic group in the framework of \cite{ST}, and provides an irreducibility criterion analogous to the classical Mackey irreducibility criterion.
	
	Mackey theory in its more classical incarnations (e.g. for representations of finite groups, or for smooth complex representations of $p$-adic groups) studies induced representations. In very broad terms, if $H\leq G$ is a suitable subgroup of a group $G$ and $V$ is an $H$-representation, then Mackey theory boils down to the following key components, with the precise formulation depending on the context:
	
	\begin{enumerate}[(i)]
		\item Mackey decomposition: Under suitable conditions, $\cind_H^GV$ decomposes as an $H$-representation into the following direct sum:
		\begin{equation*}
			\mathrm{Res}_H^G\cind_H^GV\cong \underset{g\in H\setminus G/H}{\oplus} \cind_{H\cap gHg^{-1}}^H \mathrm{Res}_{H\cap gHg^{-1}}^{gHg^{-1}}gV.
		\end{equation*}
		\item Semisimplicity: If $V$ is a semisimple $H$-representation, then so is $\mathrm{Res}_H^G\cind_H^GV$.
		\item Mackey irreducibility criterion: Suppose that $V$ is irreducible and does not appear as a constituent of any other (semisimple) Mackey summand in (i). Then $\cind_H^GV$ is an irreducible $G$-representation. Note that the condition on Mackey summands can be expressed as requiring that
		\begin{equation*}
			\mathrm{Hom}_{H\cap gHg^{-1}}(V, gV)=0
		\end{equation*}
		for all $g\in G\setminus H$, i.e. `there are no non-trivial intertwiners for $V$'.
	\end{enumerate}
	
	One straightforward incarnation of this philosophy is in the situation of (complex, finite-dimensional) representations of finite groups \cite[chapter 7]{Serre}, where semisimplicity is automatic and the irreducibility criterion is a direct consequence of Frobenius reciprocity and basic character theory.
	
	Similarly, in the case of smooth complex representations of $p$-adic groups, an analogous theory holds e.g. for $H$ an open subgroup of the form $H=H_0Z_G$ with $H_0$ compact open, see \cite[Theorem 11.4]{BH}. 
	
	This paper spells out a version of the three points above for locally analytic representations of $p$-adic groups.
	
	Our setting will be the following: Let $L/\mathbb{Q}_p$ be a finite extension and let $G$ be a locally $L$-analytic group -- for example, the $L$-valued points of some algebraic group over $L$. Let $H\leq G$ be an open subgroup which is compact mod centre in the same sense as above. 
	
	Let $K$ be a complete nonarchimedean field extension of $L$. The representations which we consider are the locally $L$-analytic $G$-representations in topological $K$-vector spaces of Schneider--Teitelbaum \cite{ST}, i.e. we consider (barrelled, Hausdorff, locally convex) topological $K$-vector spaces $V$ with a $G$-action such that for each $v\in V$, the orbit map $G\to V$, $g\mapsto g\cdot v$ is a locally $L$-analytic map. 
	
	Representations of this type are ubiquitous in the $p$-adic local Langlands programme, see e.g. \cite{Kir}, \cite{Colmezla}, \cite{EGH} and the references therein. The locally analytic compact induction is then a natural functor to obtain representations for non-compact groups like $G=\mathrm{GL}_n(\mathbb{Q}_p)$ out of representations for a compact open subgroup, like $H_0=\mathrm{GL}_n(\mathbb{Z}_p)$. In general, such compact inductions are very large if $G$ is non-compact: they are usually neither admissible nor topologically irreducible. In fact, comparable to \cite{BL} and \cite{Orlik}, one might try reducing to a fixed Hecke character, and then study irreducible constituents (but we warn that there exists ample evidence from characteristic $p$ that even after fixing a Hecke character, the representation might still be non-admissible, see e.g. \cite{Feng}). From this perspective, Mackey theory can be viewed as the starting point to investigate the shape of the corresponding Hecke algebra $\mathcal{H}=\mathrm{End}_G (\cind_H^GV)^{\mathrm{op}}$, and the irreducibility criterion discusses the case where $\mathcal{H}\cong\mathrm{End}_H(V)^{\mathrm{op}}$.
	
	We introduce a version of semisimplicity called uniform semisimplicity, which is particularly well-suited for the Fr\'echet--Stein setting, and prove the following (new terminology is explained after the Theorem):
	
	\begin{thm}\label{IntroThm}
		Let $G$ be a locally $L$-analytic group, $H=H_0Z_G\leq G$ an open subgroup which is compact mod centre, and let $V$ be an admissible locally $L$-analytic $H$-representation over $K$.
		\begin{enumerate}[(i)]
			\item Mackey decomposition: There is a canonical isomorphism of locally $L$-analytic $H$-representations
			\begin{equation*}
				\mathrm{Res}_H^G\cind_H^G V\cong \underset{g\in H\setminus G/H}{\oplus} \cind_{H\cap gHg^{-1}}^H\mathrm{Res}_{H\cap gHg^{-1}}^{gHg^{-1}}gV.
			\end{equation*}
			In particular, $\cind_H^G V$ is an ind-admissible $G$-representation.
			\item Semisimplicity: If $M=V'$ is a uniformly semisimple $D(H_0)$-module, then each Mackey summand
			\begin{equation*}
				\cind_{H\cap gHg^{-1}}^H \mathrm{Res}_{H\cap gHg^{-1}}^{gHg^{-1}}gV
			\end{equation*}
			is a topologically semisimple $H_0$-representation.
			\item Mackey irreducibility criterion: If $M=V'$ is a uniformly simple $D(H_0)$-module such that
			\begin{equation*}
				\mathrm{Hom}_{D(H\cap gHg^{-1})}(gM, M)=\begin{cases}
					K & \text{if $g\in H$}\\
					0 & \text{if $g\in G\setminus H$}
				\end{cases}
			\end{equation*} 
			then $\cind_H^GV$ is a topologically irreducible, ind-admissible $G$-representation.
		\end{enumerate}
	\end{thm}
	
	We need to briefly comment on some of the terminology in the theorem above.
	
	The notion of ind-admissible representation was introduced in \cite{Orlik} as a natural generalization of admissibility (see also \cite[Proposition 7.1.7]{KisinStrauch} for a related notion). As mentioned before, compactly inducing to a non-compact group $G$ rarely produces admissible representations. At the end of the paper, we discuss several natural examples, and we observe that in some cases, the topologically irreducible representations obtained by Theorem \ref{IntroThm} are actually admissible.
	
	The notion of `uniform semisimplicity' is a variant of topological semisimplicity, which we introduce in section 3: working dually with modules over the Fr\'echet--Stein distribution algebra $D(H_0)\cong \varprojlim D_n(H_0)$ of some compact open subgroup $H_0$ with $H=H_0Z_G$, uniform semisimplicity translates to semisimplicity on each Noetherian Banach level $D_n(H_0)$ (at least for sufficiently large $n$) in a suitably compatible manner. Thus, this condition is well-suited to develop a proof for (ii) which bypasses the topological subtleties on the Fr\'echet level. 
	
	We phrase some of our arguments involving $p$-adic functional analysis in the language of complete bornological vector spaces, remaining consistent with the appendix to \cite{Orlik}, but almost all our arguments are entirely algebraic (Corollary \ref{coadcheck} being the exception which confirms the rule). All spaces involved are extremely mild from an analytic perspective -- the modules are countable products of coadmissible modules, and in particular, (nuclear) Fr\'echet spaces. We remark in particular that the same proof can be carried out verbatim e.g. in the setting of solid modules over the distribution algebra, see \cite{Bosco}, \cite{JERC}. 
	
	The paper is organized as follows. In section 2, we recall the basic framework of ind-admissible representations, and dually, the theory of pro-coadmissible modules, building on \cite{Orlik}. Theorem \ref{IntroThm}.(i) follows immediately and was already proved in the appendix to \cite{Orlik} (see \cite[Lemma A.19]{Orlik}). In section 3, we introduce various notions of semisimplicity for modules over a distribution algebra, and study how they behave under restriction, conjugation, and induction. This provides the technical heart of the paper, and allows us to prove Theorem \ref{IntroThm}.(ii) (see Proposition \ref{topssfactors}). From there, it is easy to deduce Theorem \ref{IntroThm}.(iii) via a standard argument (see Theorem \ref{FSMackeycentral}), which we spell out in section 4. We also give a variant of the irreducibility criterion which asks for a condition on each of the Noetherian Banach levels, which in some cases may be easier to check (Theorem \ref{BanachMackey}).
	
	In section 5, we provide several examples of representations where the Mackey criterion applies: we induce characters to the $p$-adic Heisenberg group from its integral form, and we induce finite-dimensional irreducible representations to the Borel subgroup of $\mathrm{SL}_2(\mathbb{Q}_p)$ from its uniform `congruence subgroup'. We point out that in the first case, we obtain that the resulting topologically irreducible representations are even admissible, while this never occurs in the second case.

	\subsection*{Acknowledgements}
	The author would like to thank Sascha Orlik for several helpful discussions and comments on a draft version of this paper.

	\section{Ind-admissible representations and the Mackey decomposition}
	
	We presuppose knowledge of the basic theory of admissible locally analytic representations as given in \cite{ST}. Recall that a locally analytic representation of a locally $L$-analytic group $G$ is generally described as the dual of a module over the locally analytic distribution algebra $D(G):=D(G, K)$. 
	
	If $G$ is compact, then $D(G)$ is a nuclear Fr\'echet--Stein $K$-algebra (i.e. it is a Fr\'echet--Stein algebra whose underlying topological vector space is nuclear Fr\'echet). This allows us to characterize \emph{admissible} $G$-representations as those which are dual to a coadmissible $D(G)$-module. For general groups $G$, admissibility is defined by restricting to some compact open subgroup.
	
	We remark that in \cite{ST}, the additional assumption of spherical completeness is imposed on $K$ to ensure that the duality functor is well-behaved and e.g. the Hahn-Banach theorem can be applied. In our context, this restriction does not appear to be necessary, as we will only be concerned with spaces of countable type. We refer to \cite[section 4.2]{Schikhof} for a discussion of the Hahn-Banach theorem within this framework, and to \cite{Mihara} for a similar treatment concerning Banach representations. 
	
	Let $G$ be a locally $L$-analytic group and let $H_0\leq G$ be a compact open subgroup. For $H=H_0Z_G$, we will soon encounter a functor
	\begin{align*}
		\cind_H^G: & \mathrm{Rep}^{\mathrm{la}}(H)\to \mathrm{Rep}^{\mathrm{la}}(G)\\
			&	V\mapsto \cind_H^GV, 
	\end{align*}
	but it is extremely rare for $\cind_H^G V$ to be admissible, even if $V$ is. This forces us to work in a slightly larger category.
	
	We recall the following definition from \cite[Definition A.4]{Orlik}.
	
	\begin{defn}
		Let $A$ be a Fr\'echet--Stein algebra. A Fr\'echet $A$-module $M$ is called \textbf{pro-coadmissible} if $M$ can be written as a countable inverse limit of coadmissible $A$-modules, with surjective transition maps.
	\end{defn}
	
	We collect the following results from the appendix to \cite{Orlik}, which generalize the usual properties of coadmissible modules to the pro-coadmissible setting.
	
	\begin{prop}\label{procoadproperties}
		Let $A$ be a nuclear Fr\'echet--Stein $K$-algebra.
		\begin{enumerate}[(i)]
			\item The category of pro-coadmissible $A$-modules (with continuous $A$-module morphisms) is abelian, containing all coadmissible $A$-modules, and is closed under countable direct products.
			\item Any pro-coadmissible $A$-module $M$ is a nuclear Fr\'echet $K$-vector space. In particular, $M$ is reflexive as a locally convex topological vector space.
			\item Any continuous $A$-module morphism between pro-coadmissible $A$-modules is strict.
		\end{enumerate}
	\end{prop}
	
	\begin{proof}
		See \cite[Proposition A.15]{Orlik} for (i). As any coadmissible module over a nuclear Fr\'echet--Stein algebra is nuclear by an argument analogous to \cite[Proposition 5.5]{SixOp}, (ii) follows immediately, using \cite[Corollary 8.5.3, Definition 8.4.2]{Schikhof} for reflexivity. See \cite[Corollary A.13]{Orlik} for (iii).
	\end{proof}
	
	Using the same argument as in \cite[Lemma A.5]{Orlik}, we can regard the category of pro-coadmissible $A$-modules as a full subcategory of $\mathrm{Mod}_{\h{\B}c_K}(A)$, the category of complete bornological $A$-modules.
	
	We refer to \cite{SixOp} for a detailed discussion of $\h{\B}c_K$ and corresponding module categories. We will not require much from this technical framework, except that $\h{\B}c_K$ is closed symmetric monoidal with respect to the completed tensor product $\h{\otimes}_K$, giving rise to relative completed tensor products $\h{\otimes}_A$. If $A\cong \varprojlim A_n$ is a nuclear Fr\'echet--Stein algebra and $M\cong \varprojlim M_n$ is a coadmissible $A$-module, we can regard $A$ as a complete bornological $K$-algebra and $M$ as a complete bornological $A$-module such that the finitely generated $A_n$-module $A_n\otimes_A M\cong M_n$, equipped with its canonical Banach structure, is naturally isomorphic to $A_n\h{\otimes}_A M$ (see \cite[Corollary 5.38]{SixOp}).   
	
	\begin{cor}\label{coadcheck}
		Let $A\cong \varprojlim A_n$ be a nuclear Fr\'echet--Stein $K$-algebra. Let $(M^i)_{i\in I}$ be a countable family of coadmissible $A$-modules. Then $M:=\prod_{i\in I} M^i$ is a pro-coadmissible $A$-module, and the natural morphism
		\begin{equation*}
			A_n\h{\otimes}_A \prod_i M^i\to \prod_i (A_n\h{\otimes}_A M^i)
		\end{equation*}
		is an isomorphism for each $n$.
		
		In particular, $M$ is coadmissible if and only if for each $n$, 
		\begin{equation*}
			A_n\otimes_AM^i=A_n\h{\otimes}_AM^i=0
		\end{equation*}
		for all but finitely many $i\in I$.
	\end{cor}
	\begin{proof}
		The fact that $M$ is pro-coadmissible is immediate from the definition. By \cite[Lemma A.6]{Orlik}, the natural morphism
		\begin{equation*}
			A_n\h{\otimes}_A \prod_i M^i\to \prod_i (A_n\h{\otimes}_AM^i)
		\end{equation*}
		is an isomorphism for each $n$. By \cite[Corollary A.8]{Orlik}, a pro-coadmissible $A$-module $N$ is coadmissible if and only if $A_n\h{\otimes}_A N$ is finitely generated for each $n$.
		
		But now by Noetherianity of $A_n$, $\prod (A_n\h{\otimes}_AM^i)$ is a finitely generated $A_n$-module if and only if all but finitely many factors are zero, and the result follows.
	\end{proof}
	
	Now let $L/\mathbb{Q}_p$ be a finite field extension and $G$ a locally $L$-analytic group. Let $K$ be a complete nonarchimedean field extension of $L$.
	
	We say that a Fr\'echet $D(G)$-module $M$ is pro-coadmissible if it is pro-coadmissible as a $D(H_0)$-module for some compact open subgroup $H_0$. It follows from \cite[Corollary A.14]{Orlik} that this does not depend on the choice of subgroup.
	
	Dually, a locally $L$-analytic $G$-representation on a (barrelled, Hausdorff) locally convex $K$-vector space $V$ of compact type is called ind-admissible if $V'$ is a pro-coadmissible $D(G)$-module. Equivalently, $\mathrm{Res}_{H_0}^GV$ is the strict inductive limit of admissible $H_0$-subrepresentations for some (any) compact open subgroup $H_0$.
	
	Thanks to reflexivity, the duality functor induces an anti-equivalence between ind-admissible representations and pro-coadmissible modules over the distribution algebra, see \cite[Corollary A.3]{Orlik}.
	
	As in the appendix to \cite{Orlik}, we occasionally like to consider pro-coadmissible modules as complete bornological modules over $D_b(G)$, which is our notation for $D(G)$ endowed with the compactoid bornology. Explicitly,
	\begin{equation*}
		D_b(G)= \underset{g\in G/H_0}{\oplus} gD(H_0)\cong \underset{g\in G/H_0}{\oplus} D(H_0)g^{-1},
	\end{equation*} 
	which becomes a complete bornological $K$-algebra with the induced direct sum bornology.
	
	We adopt the following notation: if $M$ is a complete bornological $D(H_0)$-module and $g\in G$, then we denote by $gM$ the module over $D(gH_0g^{-1})\cong gD(H_0)g^{-1}$ obtained by twisting $M$ by $g$, i.e. the underlying complete bornological vector space of $gM$ is $M$, where we usually write $gm$ to denote the element corresponding to $m\in M$, with action 
	\begin{equation*}
		(gPg^{-1})\cdot gm=g(Pm).
	\end{equation*}
	Note that if $\widetilde{\M}$ is a complete bornological $D_b(G)$-module containing $M$ as a $D(H_0)$-submodule, then the subspace $g\cdot M\subseteq \widetilde{M}$ is a $D(gH_0g^{-1})$-submodule which is naturally isomorphic to $gM$, so that there is no ambiguity in our notation.
	
	The same convention applies to locally analytic representations.
	
	We can now introduce the key object of study for this article. Note that a typical example of an ind-admissible $G$-representation is of course any locally analytic representation $V$ such that $\mathrm{Res}_{H_0}^GV$ splits into a countable direct sum of admissible $H_0$-representations.
	
	With this in mind, we consider the following construction:
	
	Let $H_0\leq G$ be a compact open subgroup, and let $H=H_0Z_G$. For an admissible locally $L$-analytic $H$-representation $(V, \rho)$, consider the locally analytic compact induction
	\begin{equation*}
		\cind_H^GV=\left\{
		f: G\to V | \begin{matrix}\ f \text{ has compact support modulo $H$},\\ \ f(gh)=\rho(h^{-1})(f(g)) \ \forall g\in G, \ h\in H
		\end{matrix}\right\}.
	\end{equation*}
	Here, $f$ having compact support modulo $H$ means that there exists some compact subset $C$ such that $f(g)=0$ for all $g\notin CH$. As $H=H_0Z_G$ with $H_0$ compact, this is the same as having compact support modulo $Z_G$. As $H$ is open, an equivariant function $f$ has compact support modulo $H$ if and only if it is supported on finitely many $H$-cosets $gH\in G/H$.
	
	The group $G$ acts on $\cind_H^G$ via $(g*f)(g')=f(g^{-1}g')$. We can identify $\cind_H^GV$ with the locally convex $K$-vector space $\oplus_{g\in G/H} gV$, sending $f\in \cind_H^GV$ to $\sum gf(g)$. This topological structure turns $\cind_H^GV$ into a locally convex $K$-vector space, and the action above makes it a locally $L$-analytic $G$-representation. 
	
	The following easy consequence of the above was already proved in the appendix to \cite{Orlik}:
	
	\begin{prop}\label{Mackeydecomp}
		Mackey decomposition: There is a canonical isomorphism of locally $L$-analytic $H$-representations
		\begin{equation*}
			\mathrm{Res}_H^G \cind_H^G V\cong \underset{g\in H\setminus G/H}{\oplus} \cind_{H\cap gHg^{-1}}^H \mathrm{Res}_{H\cap gHg^{-1}}^{gHg^{-1}}gV.
		\end{equation*}
		In particular, $\cind_H^GV$ is an ind-admissible $G$-representation.
	\end{prop}
	\begin{proof}
		The Mackey decomposition was already established in \cite[Lemma A.19]{Orlik}. The ind-admissibility was shown in \cite[Lemma A.20]{Orlik}.
	\end{proof}	
	
	We remark that on the dual side, a corresponding direct product decomposition exists: By the description of $\cind_H^G$ as a left adjoint to restriction (\cite[Proposition 2.1]{Orlik}), we have dually the natural isomorphism
	\begin{equation*}
		(\cind_H^GV)'\cong \mathrm{Hom}_{D_b(H)}(D_b(G), V'),
	\end{equation*}
	where we regard the right-hand side as a left $D_b(G)$-module with the aid of the usual automorphism induced by $g\mapsto g^{-1}$.
	
	Thence
	\begin{align*}
		(\cind_H^GV)'&\cong \mathrm{Hom}_{D_b(H)}(D_b(G), V')\\
		&\cong \mathrm{Hom}_{D_b(H)}(\oplus_{g\in G/H} D_b(H)g^{-1}, V')\\
		&\cong \prod_{g\in G/H} \mathrm{Hom}_{D_b(H)}(D_b(H)g^{-1}, V')
	\end{align*}
	as complete bornological vector spaces (and in particular, as Fr\'echet spaces).
	
	This allows us (taking $g=e$) to consider $V'$ as a closed subspace (even a closed $D_b(H)$-submodule) of $(\cind V)'$, and the closed subspace $\mathrm{Hom}(D_b(H)g^{-1}, V')\subseteq (\cind V)'$ can then be identified with $g\cdot V'\subseteq (\cind V)'$ as a closed $D_b(gHg^{-1})$-submodule.
	
	We remark that under the duality between $\cind V$ and $(\cind V)'$, the factor $gV'$ (now viewed as a quotient of the direct product) identifies with the dual of $gV$ -- the indices match up because the action of the distribution algebra involves an inversion to produce left modules. 
	
	In this way, we have an isomorphism of Fr\'echet spaces
	\begin{equation*}
		(\cind_H^GV)'\cong \prod_{g\in G/H} gV',
	\end{equation*}
	
	and sorting the factors as in Proposition \ref{Mackeydecomp} yields an isomorphism
	\begin{equation*}
		(\cind_H^GV)'\cong \prod_{g\in H\setminus G/H} D_b(H)\otimes_{D_b(H\cap gHg^{-1})}gV'
	\end{equation*}
	of $D_b(H)$-modules.
	
	The rest of this article is devoted to the study of this module.

	\section{Semisimplicity}
	
	In most classical incarnations, Mackey theory relies on the interplay between induction, restriction, conjugation, and semisimplicity -- one crucial input is that the notion of semisimplicity is insensitive to restricting from a compact group to an open subgroup. 
	
	If $G_0$ is a compact locally $L$-analytic group and $H_0\leq G_0$ is an open subgroup, analogous statements for coadmissible $D(G_0)$-modules resp. $D(H_0)$-modules are very subtle: the most natural notion of simplicity in this context is that of being \emph{topologically} simple, but this means that standard arguments involving Zorn's lemma are not easily available. For example, it seems not clear whether a coadmissible $D(G_0)$-module which has the property that every closed submodule has a (closed) complement can be written as a direct sum of topologically simple modules. At least to the author, it is not evident that such a module contains a topologically simple submodule, or admits a topologically simple quotient if it is finitely generated.
	
	We bypass these subtleties with the help of three observations:
	\begin{enumerate}[(i)]
		\item On the Noetherian Banach level, there is no difficulty (as every submodule of a finitely generated module is closed): we study finitely generated, semisimple $D_r(G_0)$-modules in the classical sense. In this case, one can directly adapt the classical proof (e.g. \cite[Lemma 2.7]{BH}) that a finitely generated $D_r(G_0)$-module is semisimple if and only if its restriction to $D_r(H_0)$ is semisimple.
		\item For modules over the distribution algebra, we study coadmissible modules which are topologically semisimple in the sense that they are \emph{finite} direct sums of topologically simple modules. The finite length condition ensures that we do not have to invoke Zorn's lemma to find minimal and maximal closed submodules.
		\item If $M$ is a topologically semisimple $D(G_0)$-module, it is not clear that $M$ is topologically semisimple over $D(H_0)$. We can correct this by enforcing a slightly stronger condition, and require $M$ to be \emph{uniformly semisimple} over $D(G_0)$. This yields better control at the Noetherian Banach level, so that we can use our results from (i).
	\end{enumerate}
	
	This approach allows us to adapt the classical arguments to our setting (see Proposition \ref{ssFSres}), proving Theorem \ref{IntroThm}.(ii) in Proposition \ref{topssfactors}.
	
	Recall that the Fr\'echet--Stein algebra $D(G_0)$ can be written as the limit of Banach completions $\varprojlim D_r(G_0)$, where the individual norms are dependent on a choice of an $L$-uniform open normal subgroup $U$.
	
	We briefly indicate these constructions and refer the interested reader to \cite{ST} for more details. Let $U\trianglelefteq G_0$ be an $L$-uniform open normal subgroup in the sense of \cite[Remark 2.2.5.(ii)]{OS}. In particular, $U$ is a uniform pro-$p$ group and as a $\mathbb{Q}_p$-manifold, $U$ is isomorphic to $\mathbb{Z}_p^d$ for some $d\geq 0$. Now $D^{\mathbb{Q}_p\mathrm{-la}}(\mathbb{Z}_p^d, K)$ is isomorphic to the $K$-algebra of analytic functions on a $d$-dimensional open polydisk $\mathfrak{X}^d$ of radius $1$, and we thus obtain an isomorphism of Fr\'echet $K$-vector spaces
	\begin{equation*}
		D^{\mathbb{Q}_p\mathrm{-la}}(U, K)\cong \O(\mathfrak{X}^d).
	\end{equation*}
	For a chosen topological basis $h_1, \hdots, h_d$ of $U$ as a uniform pro-$p$ group, the isomorphism above sends $h_i$ to $1+x_i$, where $x_1, \hdots, x_d$ are the variables on $\mathfrak{X}^d$.
	
	The Fr\'echet structure on the right is determined by writing $\O(\mathfrak{X}^d)$ as the inverse limit of functions on closed polydiscs with radii $r<1$. It was shown in \cite[Proposition 4.2]{ST} that if $1/p\leq r<1$, then the induced norm $|-|_r$ on $D^{\mathbb{Q}_p\mathrm{-la}}(U, K)$ is a submultiplicative algebra norm, whose Banach completion $D_r^{\mathbb{Q}_p}(U, K)$ is a Noetherian Banach algebra, which as a Banach $K$-vector space is isomorphic to a Tate algebra in $d$ variables.
	
	Writing $D(U):=D^{L\mathrm{-la}}(U,K)$ as a quotient of $D^{\mathbb{Q}_p\mathrm{-la}}(U, K)$, this yields the Fr\'echet--Stein presentation of $D(U)\cong \varprojlim D_r(U)$. For the distribution algebra $D(G_0)$, we note that $D(G_0)$ is a finite free $D(U)$-module (both on the left and the right), with a basis given by coset representatives $g_1, \hdots, g_s$ for $G_0/U$. In this way, \cite[Theorem 5.1]{ST} produces the Fr\'echet--Stein presentation for $D(G_0)$ as
	\begin{equation*}
		D(G_0)\cong \varprojlim D_r(G_0),
	\end{equation*}
	where $D_r(G_0)\cong  \oplus_{i=1}^s D_r(U) g_i^{-1}$ (where we use inverses purely for compatibility with our earlier notation).
	
	The following notation will be convenient: if $r$ happens to be of the form $r=p^{-1/p^n}$ for some $n\geq 0$, then \cite[Theorem 6.5.11]{Ardakov} yields an isomorphism of $K$-Banach algebras
	\begin{equation*}
		D_r(U)\cong D_{1/p}(U^{p^n})\rtimes_{U^{p^n}}U,
	\end{equation*}
	where we refer to \cite[Definitions 2.2.1, 2.2.3]{Ardakov} for the definition of the crossed product algebra on the right hand side. The key takeaway for us is the fact that $D_r(U)$ is free of finite rank over $D_{1/p}(U^{p^n})$, so that the $r=p^{-1/p^n}$-norm is simply induced by the $1/p$-norm for a smaller subgroup, $U^{p^n}$.
	
	In this way, we can now rewrite the above to
	\begin{equation*}
		D_r(G_0)\cong D_{1/p}(U^{p^n})\rtimes_{U^{p^n}} G_0.
	\end{equation*}
	Realizing $D_{1/p}(U^{p^n})$ as a certain Banach completion of the enveloping algebra of the Lie algebra, this can also already be pieced together from \cite[Theorem 1.4.2]{Kohlhaase}.
	
	Now recall that $H_0\leq G_0$ is some open subgroup.
	
	Since $U^{p^m}$ is contained in $H_0$ for some $m$ (as the $U^{p^m}$ form a neighbourhood basis of $e$), we can choose norms for the Fr\'echet--Stein algebras $D(G_0)$ and $D(H_0)$ compatibly, i.e. for $r=p^{-1/p^n}$ with $n>m$, we can write
	\begin{equation*}
		D_r(H_0)\cong D_{1/p}(U^{p^n})\rtimes_{U^{p^n}} H_0
	\end{equation*} 
	and
	\begin{equation*}
		D_r(G_0)\cong D_{1/p}(U^{p^n})\rtimes_{U^{p^n}}G_0,
	\end{equation*}
	so that $D_r(H_0)$ is naturally contained in $D_r(G_0)$, making $D_r(G_0)$ a finite free left and right $D_r(H_0)$-module, with a basis given by suitable coset representatives.
	
	From now on, the notation $D_r(G_0)$ will always presuppose that we have fixed an $L$-uniform open normal subgroup $U\leq G_0$, and $r$ will always be of the form $r=p^{-1/p^n}$ for some $n\geq 0$.
	
	As usual, we say that an abstract $D_r(G_0)$-module $M$ is \textbf{semisimple} if any submodule admits a complement. It is very well-known that this is equivalent to $M$ being the direct sum of simple modules, see \cite[Theorem 2.4]{Lam}.
	
	\begin{prop}
		\label{ssres}
		Let $U\leq G_0$ be an $L$-uniform open normal subgroup, and let $m$ be such that $U^{p^m}\subseteq \underset{g\in G_0/H_0}{\cap} gH_0g^{-1}$. Let $r=p^{-1/p^n}$ with $n>m$.
		
		Let $M$ be an abstract $D_r(G_0)$-module. Then $M$ is semisimple if and only if it is semisimple as a $D_r(H_0)$-module.
	\end{prop}
	\begin{proof}
		This is adapted from \cite[Lemma 2.7]{BH}, where the same result is shown for smooth complex representations.
		
		Since $H_0$ is open, it is of finite index in $G_0$, with $|G_0/H_0|=d$, say.
		
		Suppose first that $M$ is semisimple over $D_r(H_0)$. Let $N$ be a $D_r(G_0)$-submodule of $M$. Since $M$ is semisimple over $D_r(H_0)$, there exists a $D_r(H_0)$-submodule $N'\subseteq M$ such that $M\cong N\oplus N'$. Let $\pi: M\to N$ denote the projection map, and consider the stabilisation
		\begin{align*}
			\pi^{G_0}:& M\to N\\
			&m\mapsto \frac{1}{n}\sum_{g\in G_0/H_0} g\pi(g^{-1}m) 
		\end{align*} 
		Since $D_r(G_0)$ is a free finite $D_r(H_0)$-module with basis given by coset representatives of $G_0/H_0$ and $\pi$ is $D_r(H_0)$-linear, it follows that $\pi^{G_0}$ is a $D_r(G_0)$-linear map, and composition with the embedding $N\to M$ thus yields an idempotent, $D_r(G_0)$-linear map with image $N$. This exhibits $\mathrm{ker} (\pi^{G_0})$ as a complement to $N$, and $M$ is semisimple as a $D_r(G_0)$-module.
		
		Conversely, let $M$ be a semisimple $D_r(G_0)$-module.
		
		Replacing $H_0$ by $\cap_{g\in G_0/H_0} gH_0g^{-1}$ and using the direction above, it suffices to consider the case where $H_0$ is a normal subgroup. 
		
		It also clearly suffices to consider the case where $M$ is simple. Since $M$ is cyclic over $D_r(G_0)$, it is finitely generated over $D_r(H_0)$ -- in fact, any $m\neq M$ generates $M$ as a $D_r(G_0)$-module, and then $g_1m, \hdots, g_dm$ generate $M$ as a $D_r(H_0)$-module, where $g_1, \hdots, g_d$ is a set of representatives for $G_0/H_0$. Therefore, $M$ admits a simple $D_r(H_0)$-module quotient $N$. 
		
		The natural map $M\to N$ thus induces a non-trivial $D_r(G_0)$-module morphism $M\to \mathrm{Hom}_{D_r(H_0)}(D_r(G_0), N)$. As a $D_r(H_0)$-module, $\mathrm{Hom}_{D_r(H_0)}(D_r(G_0), N)$ is isomorphic to $\oplus_{g\in G_0/H_0} gN$ and hence semisimple: as we assume $H_0$ to be normal, $gN$ is a $D_r(H_0)$-module, which is moreover simple by simplicity of $N$.
		
		As $M$ was assumed to be simple over $D_r(G_0)$, the morphism $M\to \mathrm{Hom}(D_r(G_0), N)$ is injective, so $M$ is a $D_r(H_0)$-submodule of a semisimple module, and thus semisimple. 
	\end{proof}
	
	We now formulate different versions of semisimplicity for coadmissible modules over a Fr\'echet--Stein algebra. 
	
	\begin{defn}
		Let $A$ be a complete bornological $K$-algebra. A Fr\'echet $A$-module $M$ is called \textbf{topologically simple} if it has no non-trivial closed submodule.
		
		If $A$ is Fr\'echet--Stein, then a coadmissible $A$-module is called \textbf{topologically semisimple} if it is of finite length in $\C_A$ (the category of coadmissible $A$-modules) and each closed submodule has a complement. 
	\end{defn}
	
	For the notion of topological simplicity, we allow $A$ to be any complete bornological algebra rather than a Fr\'echet--Stein algebra, so that we can also talk of topologically simple $D_b(G)$-modules for $G$ non-compact.
	
	We also remark that if $A$ is a Fr\'echet--Stein algebra, then a topologically simple coadmissible $A$-module is the same as a simple object in the category $\C_A$, thanks to \cite[Lemma 3.6]{ST}.
	
	Recall from \cite[Lemma 3.9]{ST} that a coadmissible $D(G_0)$-module $M$ which has the property that each $D_{r_m}(G_0)\otimes_{D(G_0)}M$ is simple for a sequence $r_m$ tending to $1$, is topologically simple -- in fact, it is even simple as an abstract $D(G_0)$-module. The converse seems far from clear and quite possibly false. To exploit techniques both on the Fr\'echet and on the Banach level, we therefore make the following definition:
	
	\begin{defn}
		Let $A\cong\varprojlim A_n$ be a Fr\'echet--Stein $K$-algebra. A coadmissible $A$-module $M\cong\varprojlim M_n$ is called \textbf{uniformly simple} if $M_n$ is a simple $A_n$-module for all sufficiently large $n$.
		
		A coadmissible $A$-module $M$ is called \textbf{uniformly semisimple} if $M$ is the finite direct sum of uniformly simple $A$-modules.
	\end{defn}  
	
	We warn the reader that the notion of uniform (semi-)simplicity does depend on the choice of Fr\'echet--Stein presentation $A\cong \varprojlim A_n$.

	The following are easy topological analogues of basic properties of semisimple modules.
	
	\begin{lem}\label{closedss}
		Closed submodules of topologically semisimple $A$-modules are topologically semisimple.
	\end{lem}
	\begin{proof}
		The finite length is immediate, and the existence of complements is exactly the same argument as in the algebraic case: If $S\subseteq N\subseteq M$ is a sequence of closed $A$-submodules and $M$ is topologically semisimple, write $M\cong S\oplus S'\cong N\oplus N'$ and show that $N\cong S\oplus (N\cap S')$.
	\end{proof}
	
	\begin{lem}\label{ssdirectsum}
		A coadmissible $A$-module is topologically semisimple if and only if it is a finite direct sum of topologically simple $A$-modules.
	\end{lem}
	\begin{proof}
		Let $M$ be a topologically semisimple $A$-module. As $M$ is of finite length in $\C_A$, it has a finite filtration by coadmissible (therefore closed) $A$-submodules $M^i$, with topologically simple subquotients. By Lemma \ref{closedss}, the short exact sequence
		\begin{equation*}
			0\to M^{i-1}\to M^i\to M^i/M^{i-1}\to 0
		\end{equation*}  
		splits for all $i$, so $M$ is a finite direct sum of topologically simple modules.
		
		Conversely, if $M=\oplus_{i\in I} M^i$ is a finite direct sum of topologically simple $A$-module $M^i$, then it is of finite length in $\C_A$. If $N\subseteq M$ is any closed submodule, let $J\subseteq I$ be a maximal subset such that 
		\begin{equation*}
			N\cap (\oplus_{i\in J} M^j)=\{0\}.
		\end{equation*} 
		Thus the natural morphism $N\oplus (\oplus_JM^j)\to M$ is injective, and we can regard $N':=N\oplus (\oplus_JM^j)$ as a submodule of $M$. As $N'$ is coadmissible, it is closed in $M$, so for each $i\in I$, $N'\cap M_i=0$ or $N'\cap M_i=M_i$. The maximality of $J$ then forces $N'\cap M_i=M_i$ for all $i$, i.e. $N'=M$, and $\oplus_JM^j$ is a complement to $N$. 
	\end{proof}
	
	It follows in particular that the finite direct sum of topologically semisimple modules is again topologically semisimple.
	
	\begin{lem}\label{uniftop}
		If $M$ is a coadmissible $A$-module which is uniformly semisimple (relative to some Fr\'echet--Stein presentation of $A$), then $M$ is topologically semisimple.
	\end{lem}
	\begin{proof}
		This is immediate, as uniformly simple implies topologically simple.
	\end{proof}
	
	We can now prove a Fr\'echet--Stein version of Proposition \ref{ssres}.
	
	\begin{prop}\label{ssFSres}
		Let $G_0$ be a compact locally $L$-analytic group, and let $H_0\leq G_0$ be an open subgroup. Let $M$ be a coadmissible $D(G_0)$-module.
		\begin{enumerate}[(i)]
			\item If $M$ is topologically semisimple over $D(H_0)$, then $M$ is topologically semisimple over $D(G_0)$.
			\item If $M$ is uniformly semisimple over $D(G_0)$ (with respect to some Fr\'echet--Stein presentation $D(G_0)\cong \varprojlim D_r(G_0)$), then $M$ is topologically semisimple over $D(H_0)$.
		\end{enumerate}
	\end{prop}
	
	\begin{proof}
		\begin{enumerate}[(i)]
			\item Let $M$ be a topologically semisimple over $D(H_0)$. Since any closed $D(G_0)$-submodule is also a closed $D(H_0)$-submodule, it is immediate that $M$ is of finite length in $\C_{D(G_0)}$. It remains to verify that every closed $D(G_0)$-submodule $N$ has a complement, for which we can use the same argument as in Proposition \ref{ssres}.
			\item[(ii)] As we have already proven (i), we can replace $H_0$ by $\cap gH_0g^{-1}$ as before, and assume that $H_0$ is normal in $G_0$. Assume that $M$ is uniformly simple over $D(G_0)$, the general case will then follow by taking finite direct sums.
			
			Note that 
			\begin{equation*}
				D_r(G_0)\otimes_{D(G_0)}M\cong D_r(H_0)\otimes_{D(H_0)}M,
			\end{equation*} 
			e.g. by our description of the distribution algebras at the beginning of the section. By assumption, $M_r:=D_r(G_0)\otimes_{D(G_0)}M$ is a simple $D_r(G_0)$-module, and therefore semisimple over $D_r(H_0)$ by Proposition \ref{ssres}.
			
			We can even bound the length of $M_r$ as a $D_r(H_0)$-module, independently of $r$. Let $|G_0/H_0|=d$, and let $g_1, \hdots, g_d$ denote a set of coset representatives. For any simple $D_r(H_0)$-submodule $N\subseteq M_r$, the translates $g_iN$ are also simple $D_r(H_0)$-submodules (as $H_0$ is normal), and $\sum g_iN=M_r$ by simplicity of $M_r$ over $D_r(G_0)$. It follows that $\mathrm{length}_{D_r(H_0)}(M_r)\leq d$ for each $r$.
			
			We can thus apply \cite[Lemma 4.8.(ii)]{Reichardt} to deduce that $M$ has finite length in $\C_{D(H_0)}$.
			
			In particular, $M$ admits a topologically simple quotient $N$ in $\C_{D(H_0)}$, yielding a non-zero $D(G_0)$-linear map $\theta: M\to \mathrm{Hom}_{D(H_0)}(D(G_0), N)\cong \oplus g_i N$. As $\theta$ is a morphism of coadmissible modules, it is continuous with respect to the Fr\'echet structures. Thus the kernel is closed, so that the topological simplicity of $M$ over $D(G_0)$ forces the morphism to be injective. Moreover, the image of $\theta$ is closed, as morphisms of coadmissible moduels are strict, so $\theta$ realizes $M$ as a closed $D(G_0)$-submodule of $\oplus g_iN$. 
			
			Now $\oplus g_iN$ is a topologically semisimple $D(H_0)$-module, as each $g_iN$ is topologically simple. By Lemma \ref{closedss}, $M$ is then topologically semisimple over $D(H_0)$.   
		\end{enumerate}
	\end{proof}
	
	\begin{prop}\label{topssfactors}
		Semisimplicity: Let $G$ be a locally $L$-analytic group, let $H_0\leq G$ be a compact open subgroup, and let $H=H_0Z_G$. If $V$ is an admissible locally $L$-analytic $H$-representation such that $M=V'$ is uniformly semisimple over $D(H_0)$, then
		\begin{equation*}
			D(H)\otimes_{D(H\cap gHg^{-1})}gM
		\end{equation*}
		is a topologically semisimple $D(H_0)$-module for any $g\in G$.
	\end{prop}
	\begin{proof}
		Since $H_0/(H_0\cap gH_0g^{-1})$ surjects onto $H/(H\cap gHg^{-1})$, the module
		\begin{equation*}
			D(H)\otimes_{D(H\cap gHg^{-1})}gM
		\end{equation*}
		is a quotient of 
		\begin{equation*}
			D(H_0)\otimes_{D(H_0\cap gH_0g^{-1})}gM,
		\end{equation*}
		so it suffices to prove the topological semisimplicity of the latter.
		
		We can assume without loss of generality that $M$ is uniformly simple over $D(H_0)$ (with respect to some Fr\'echet--Stein presentation associated to some open normal $U\leq H$ and some $r_m$ as before). Then $gM$ is uniformly simple over $D(gH_0g^{-1})$ (with respect to $gUg^{-1}$ and $r_m$). 
		
		Let
		\begin{equation*}
			H':=\underset{h\in H_0/H_0\cap gH_0g^{-1}}{\cap} h(H_0\cap gH_0g^{-1})h^{-1},
		\end{equation*}
		an open normal subgroup of $H_0$, contained in $gH_0g^{-1}$. By \ref{ssFSres}.(ii), $gM$ is a topologically semisimple $D(H')$-module.
		
		If $h\in H_0$, then $hgM$ is a topologically semisimple module over $D(H')\cong hD(H')h^{-1}$, as $H'$ is normal in $H_0$.
		
		Since $D(H_0)$ is free of finite rank over $D(H_0\cap gH_0g^{-1})$, it follows that
		\begin{equation*}
			D(H_0)\otimes_{D(H_0\cap gH_0g^{-1})}gM\cong \underset{h\in H_0/H_0\cap gH_0g^{-1}}{\oplus} hgM
		\end{equation*}
		is a finite direct sum of topologically semisimple $D(H')$-modules, and thus topologically semisimple over $D(H')$. We can thus apply Proposition \ref{ssFSres}.(i) to deduce that the module is topologically semisimple over $D(H_0)$.
	\end{proof}

	\section{The irreducibility criterion}
	
	In this section, let $G$ denote a locally $L$-analytic group, let $H_0\leq G$ be a compact open subgroup, and $H=H_0Z_G$. 
	
	The discussion in the previous section now allows us to formulate two irreducibility criteria for compactly induced representations, one on the Fr\'echet--Stein level, and one on the Noetherian Banach level.
	
	\begin{thm}[{Mackey irreducibility criterion, Fr\'echet version}]\label{FSMackey}
		Let $V$ be an admissible locally $L$-analytic $H$-representation such that $M=V'$ is a uniformly simple $D(H_0)$-module (relative to some Fr\'echet--Stein presentation). Assume that
		\begin{equation*}
			\mathrm{Hom}_{D(H_0\cap gH_0g^{-1})}(gM, M)=\begin{cases}
				K \ \text{if $g\in H$}\\
				0 \ \text{if $g\in G\setminus H$}.
			\end{cases}
		\end{equation*}
		Then $\mathrm{Hom}_{D_b(H)}(D_b(G), M)$ is a topologically simple $D_b(G)$-module. In particular, its dual $\cind_H^G V$ is a topologically irreducible $G$-representation.
	\end{thm}
	\begin{proof}
		Label double coset representatives for $H\setminus G/H$ as $g_1=e, g_2, \hdots$
		
		Note that
		\begin{align*}
			(\cind_H^GV)'&\cong (\oplus_{g\in H\setminus G/H} \cind_{H\cap gHg^{-1}}^H \mathrm{Res}_{H\cap gHg^{-1}}^{gHg^{-1}}gV)'\\
			&\cong \prod_i D(H)\otimes_{D(H\cap g_iHg_i^{-1})}g_iM\\
			&\cong \varprojlim_m \oplus_{i=1}^m D(H)\otimes_{D(H\cap g_iHg_i^{-1})}g_iM
		\end{align*}
		as complete bornological $D_b(H)$-modules, and in particular as $D(H_0)$-modules. 
		
		Set $S_i=D(H)\otimes_{D(H\cap g_iHg_i^{-1})}g_iM$, so that $(\cind_H^GV)'\cong \varprojlim_m \oplus_{i=1}^m S_i$. By Proposition \ref{topssfactors}, each $S_i$ is a topologically semisimple $D(H_0)$-module. Note that $S_1=M$.
		
		Now let $N\subseteq (\cind_H^GV)'$ be a non-zero closed $D_b(G)$-submodule, yielding by adjunction a non-trivial morphism $N\to M$. Let $N_m$ be the image of $N$ under the projection map
		\begin{equation*}
			(\cind_H^GV)'\to \oplus_{i=1}^mS_i.
		\end{equation*} 
		 
		As $(\cind_H^GV)'$ is pro-coadmissible by Proposition \ref{Mackeydecomp}, so is $N$ by \cite[Corollary A.2]{Orlik}. As a finite direct sum of coadmissible modules, $\oplus_{i=1}^m S_i$ is coadmissible. In particular, $N_m$ is closed in $\oplus_{i=1}^m S_i$ by Proposition \ref{procoadproperties}, and $N\cong \varprojlim_m N_m$.
		
		Since $N_m$ is a closed $D(H_0)$-submodule of $\oplus_{i=1}^m S_i$, it is in particular topologically semisimple. By construction, the projection map to $S_1=M$ yields a non-trivial morphism $N_m\to M$, which thus allows for a section $M\to N_m$.
		
		But by our assumption,
		\begin{equation*}
			\mathrm{Hom}_{D(H_0)}(\oplus_{i=1}^m S_i, M)=K,
		\end{equation*} 
		so that $M$ appears as a summand only once (namely, as the summand $S_1$). Therefore, if $N_m$ contains a closed submodule which is isomorphic to $M$, it must contain $S_1$.
		
		Taking the limit, it follows that $N$ contains $S_1=g_1M\subseteq (\cind_H^GV)'$, which topologically generates the module, and hence $N=(\cind_H^GV)'$, as required.
	\end{proof}
	
	We remark that due to the duality between admissible representations and coadmissible modules, the condition on Hom spaces can also be checked for the corresponding representations, i.e. one may consider
	\begin{equation*}
		\mathrm{Hom}_{H_0\cap gH_0g^{-1}}(V, gV)
	\end{equation*}
	instead of $\mathrm{Hom}(gM, M)$. As a consequence of our semisimplicity results, one may also switch the order and consider $\mathrm{Hom}(gV, V)$ resp. $\mathrm{Hom}(M, gM)$. Note however that the condition of uniform simplicity is expressed most naturally in terms of modules rather than representations.
	
	We point out that the assumption
	\begin{equation*}
		\mathrm{Hom}_{D(H_0)}(M, M)=K
	\end{equation*}
	implies in particular that $Z_G$ acts via a character $Z_G\to K^\times$ on $M$ (and thus likewise on $V$). This allows us to strengthen Theorem \ref{FSMackey} in the following way, replacing $H_0\cap gH_0g^{-1}$ by $H\cap gHg^{-1}$ in the intertwining condition.
	
	\begin{thm}[Mackey irreducibility criterion, strengthened Fr\'echet version]\label{FSMackeycentral}
		Let $V$ be an admissible locally $L$-analytic $H$-representation such that $M=V'$ is a uniformly simple $D(H_0)$-module (relative to some Fr\'echet--Stein presentation). Assume that
		\begin{equation*}
			\mathrm{Hom}_{D_b(H\cap gHg^{-1})}(gM, M)=\begin{cases}
				K \text{ if $g\in H$}\\
				0 \text{ if $g\in G\setminus H$}.
			\end{cases}
		\end{equation*}
		Then $\mathrm{Hom}_{D_b(H)}(D_b(G), M)$ is a topologically simple $D_b(G)$-module, with the same central character as $M$.
	\end{thm}
	
	We remark that in comparison to Theorem \ref{FSMackey}, the condition on intertwining is weaker, as $H\cap gHg^{-1}$ bigger than $H_0\cap gH_0g^{-1}$ (and usually also bigger than $(H_0\cap gH_0g^{-1})Z_G$!), so it may be easier to establish the non-existence of certain morphisms.
	
	\begin{proof}
		Let the action of $Z_G$ on $M$ be given by $\chi: Z_G\to K^{\times}$.
		
		The proof is now exactly the same argument as before, with the additional observation that each of our summands $S_i=D(H)\otimes_{D(H\cap g_iHg_i^{-1})}g_iM$ is a topologically semisimple $D(H_0)$-module on which $Z_G$ acts via $\chi$, and thus a finite direct sum of topologically simple $D_b(H)$-modules.
	\end{proof}
	
	In some situations, it is easier to understand the corresponding modules on the Noetherian Banach level. We therefore formulate a slightly stronger condition which, despite its rather cumbersome notation, has the advantage of being checkable on each Noetherian level separately. Moreover, on each level, we only need to verify the condition for finitely many double coset representatives.
	
	\begin{thm}[{Mackey irreducibility criterion, Noetherian Banach version}]
		\label{BanachMackey}
		Let $H_0\leq G$ be a compact open subgroup and let $H=H_0Z_G$.
		
		Let $V$ be an admissible locally analytic $H$-representation and set $M=V'$. List the double coset representatives $g_1=e, g_2, \hdots$ of $H\backslash G/H$ and let $U\leq H_0$ be an $L$-uniform open normal subgroup. Let $n_1< n_2< \hdots$ be a sequence of non-negative integers such that
		\begin{equation*}
			U^{p^{n_m}}\subseteq H_0\cap g_m^{-1}H_0g_m
		\end{equation*}
		and let $D_{r_m}(H_0)$ denote the corresponding Banach completion of the distribution algebra, with $r_m=p^{-1/p^{n_m}}$.
		
		We write accordingly $M_m:=D_{r_m}(H_0)\otimes_{D(H_0)} M$.
		
		For all $m\geq i\geq 1$, let $n'_{i, m}\geq n_m$ be such that 
		\begin{equation*}
			g_iU^{p^{n'_{i, m}}}g_i^{-1}\subseteq U^{p^{n_m}},
		\end{equation*}
		so that $D_{s_{i, m}}(H_0\cap g_i^{-1}H_0g_i)$ makes sense for $s_{i, m}=p^{-1/p^{n'_{i, m}}}$, and there is a natural map
		
		\begin{equation*}
			g_iD_{s_{i, m}}(H_0\cap g_i^{-1}H_0g_i)g_i^{-1}\to D_{r_m}(H_0).
		\end{equation*}

		Set $M_{i, m}=D_{s_{i, m}}(H_0)\otimes_{D(H_0)} M$.
		
		Assume that the following conditions are satisfied for each $m\geq 1$:
		\begin{enumerate}[(i)]
			\item $M_m$ is a simple $D_{r_m}(H_0)$-module.
			\item For $i=1, \hdots, m$, we have
			\begin{equation*}
				\mathrm{Hom}_{g_iD_{s_{i, m}}(H_0\cap g_i^{-1}H_0g_i)g_i^{-1}}(g_iM_{i, m}, M_m)=\begin{cases}
					K \text{ if $i=1$}\\
					0 \text{ otherwise.}
				\end{cases}
			\end{equation*} 
		\end{enumerate}
		Then $\mathrm{Hom}_{D_b(H)}(D_b(G), M)$ is a topologically simple $D_b(G)$-module. In particular, its dual $\cind_H^GV$ is a topologically irreducible $G$-representation.
	\end{thm}
	\begin{proof}
		We verify the conditions from Theorem \ref{FSMackey}. By definition, $M=\varprojlim M_m$ is uniformly simple over $D(H_0)$ relative to $D_{r_m}(H_0)$.
		
		For a fixed $i$, we have
		\begin{align*}
			\mathrm{Hom}_{D(H_0\cap g_iH_0g_i^{-1})}(g_iM, M)&\cong \varprojlim_m \mathrm{Hom}_{D(H_0\cap g_iH_0g_i^{-1})}(g_iM, M_m)\\
			&\cong \varprojlim_m \mathrm{Hom}_{g_iD(H_0\cap g_i^{-1}H_0g_i)g_i^{-1}}(g_iM, M_m).
		\end{align*}
		
		Write $H'_i=H_0\cap g_i^{-1}H_0g_i$.
		
		By tensor-hom adjunction, we have
		\begin{align*}
			\mathrm{Hom}_{g_iD(H'_i)g_i^{-1}}(g_iM, M_m)&\cong \mathrm{Hom}_{g_iD_{s_{i, m}}(H'_i)g_i^{-1}}(g_iD_{s_{i, m}}(H'_i)g_i^{-1}{\otimes}_{g_iD(H'_i)g_i^{-1}}g_iM, M_m)\\
			&\cong \mathrm{Hom}_{g_iD_{s_{i, m}}(H'_i)g_i^{-1}}(g_i(D_{s_{i, m}}(H'_i){\otimes}_{D(H'_i)}M), M_m)\\
			&\cong \mathrm{Hom}_{g_iD_{s_{i, m}}(H'_i)g_i^{-1}}(g_i(D_{s_{i, m}}(H_0){\otimes}_{D(H_0)}M), M_m)\\
			&\cong \mathrm{Hom}_{g_iD_{s_{i, m}}(H'_i)g_i^{-1}}(g_iM_{i, m}, M_m).
		\end{align*}
		
		More concretely, the first isomorphism is tensor-hom adjunction together with the fact that $M_m$ is a $D_{r_m}(H_0)$-module, and hence a $g_iD_{s_{i,m}}(H'_i)g_i^{-1}$-module by construction. The second isomorphism is the straightforward observation that twisting by $g_i$ respects tensor products via conjugation, while the third one follows from our description of Banach completions of the distribution algebra at the beginning of section 3. The fourth isomorphism is immediate by the definition of $M_{i, m}$.
		
		For $i=1$, we obtain $K$, while for $i>1$, we obtain $0$ by assumption, and taking the limit as $m$ goes to infinity, we can apply Theorem \ref{FSMackey} to conclude.
	\end{proof}

	We reiterate that the representation $\cind_H^GV$ is in general not admissible. Rather, establishing admissibility is equivalent to verifying the condition in Corollary \ref{coadcheck}.

\section{Examples}

We conclude this paper by giving two examples. Firstly, we consider representations of the Heisenberg group over $\mathbb{Q}_p$ induced by a character of the Heisenberg group over $\mathbb{Z}_p$ (and extended to the full centre). The Mackey criterion ensures that we obtain topologically irreducible representations, and a calculation exploiting the non-compactness of the centre ensures that the representations are even admissible. 

Secondly, we study representations of the Borel subgroup of $\mathrm{SL}_2(\mathbb{Q}_p)$ induced from a uniform `congruence subgroup' $H_0$. As before, we obtain topologically irreducible representations, but this time, any representation obtained in this way is far from admissible: in fact, it contains a certain $H_0$-representation with infinite multiplicity. 

It would be very interesting to establish general criteria when a compact induction is both topologically irreducible and admissible. We present the calculations below both as a first indication what kind of features might be involved, and to showcase the vastly differing outcomes afforded by this theory. 

\subsection{The $p$-adic Heisenberg group}

We consider the following setup:
\begin{equation*}
	G=\begin{pmatrix}
		1&\mathbb{Q}_p&\mathbb{Q}_p\\
		0&1&\mathbb{Q}_p\\
		0&0&1
	\end{pmatrix}, \ \ H_0=\begin{pmatrix}
	1&\mathbb{Z}_p&\mathbb{Z}_p\\
	0 & 1&\mathbb{Z}_p\\
	0&0&1
	\end{pmatrix}.
\end{equation*}

Note that
\begin{equation*}
	Z_G=\begin{pmatrix}
		1&0&\mathbb{Q}_p\\
		0&1&0\\
		0&0&1
	\end{pmatrix}\cong (\mathbb{Q}_p, +),
\end{equation*}
and $G/Z_G\cong (\mathbb{Q}_p^2, +)$. We remark that $H=H_0Z_G$ is normal in $G$.

Given a locally analytic character $\chi=(\chi_1, \chi_2): (\mathbb{Z}_p^2, +)\to K^\times$, we obtain a locally analytic character $K_\chi$ of $H_0$ by letting $H_0\cap Z_G$ act trivially and writing $H_0/H_0\cap Z_G\cong \mathbb{Z}_p^2$. 

Note that $Z_G/H_0\cap Z_G\cong \mathbb{Q}_p/\mathbb{Z}_p$, so we can extend the above to a locally analytic character of $H=H_0Z_G$ by choosing a character $\mu: \mathbb{Q}_p/\mathbb{Z}_p\to K$. We assume now that $K$ contains all $p^n$th roots of unity and choose $\mu$ to be injective (i.e. $\mu(p^{-1})$, $\mu(p^{-2}), \hdots$ form a compatible system of $p$th, $p^2$th, $\hdots$ primitive roots of unity). Regarding $\mu$ as a locally analytic character on the additive group $\mathbb{Q}_p\cong Z_G$, this ensures that $\mathrm{ker}\mu=H_0\cap Z_G$. We write $\widetilde{\chi}=(\chi, \mu)$ for the corresponding character on $H$. 

We denote the resulting $1$-dimensional $H$-representation by $K_{\chi, \mu}$ and consider the induced representation
\begin{equation*}
	V_{\chi, \mu}:=\cind_H^G K_{\chi, \mu}.
\end{equation*}

\begin{prop}
	The locally analytic $G$-representation $V_{\chi, \mu}$ is topologically irreducible, with central character given by $\mu$.
\end{prop}
\begin{proof}
	We verify the condition in Theorem \ref{FSMackeycentral}. As $H$ is normal, it suffices to show the following: for each $g\in G\setminus H$, there exists an $h\in H$ such that $\widetilde{\chi}(h)\neq \widetilde{\chi}(ghg^{-1})$.
		
	Let
	\begin{equation*}
		g=\begin{pmatrix}
			1&a&b\\
			0&1&c\\
			0&0&1
		\end{pmatrix}\in G\setminus H.
	\end{equation*}	
	Thus by assumption, either $a\notin \mathbb{Z}_p$ or $b\notin \mathbb{Z}_p$. Suppose that $a\notin \mathbb{Z}_p$. Then set
	\begin{equation*}
		h=\begin{pmatrix}
			1&0&1\\
			0&1&1\\
			0&0&1
		\end{pmatrix}
	\end{equation*}
	and note that
	\begin{equation*}
		ghg^{-1}=\begin{pmatrix}
			1&0&1+a\\
			0&1&1\\
			0&0&1
		\end{pmatrix}.
	\end{equation*}
	We then have $\widetilde{\chi}(h)=\chi_2(1)\cdot \mu(1)=\chi_2(1)$ and $\widetilde{\chi}(ghg^{-1})=\chi_2(1)\cdot \mu(1+a)\neq \chi_2(1)$, since $\mathrm{ker}\mu=\mathbb{Z}_p$.
	
	Similarly, if $b\notin \mathbb{Z}_p$, take 
	\begin{equation*}
		h=\begin{pmatrix}
			1&1&1\\
			0&1&0\\
			0&0&1
		\end{pmatrix}
	\end{equation*}
	for an analogous conclusion.
	
	Hence $\mathrm{Hom}_{D(H)}(gK_{\chi, \mu}', K_{\chi, \mu}')=0$ if $g\notin H$, and $\mathrm{Hom}_{D(H)}(K_{\chi, \mu}', K_{\chi, \mu}')=K$, as $K_{\chi, \mu}$ is $1$-dimensional.
	
	Thus the Mackey condition from Theorem \ref{FSMackeycentral} is satisfied, and $\cind_H^G K_{\chi, \mu}$ is topologically irreducible.
\end{proof}

This example also highlights once more the difference between Theorem \ref{FSMackey} and \ref{FSMackeycentral}: while $H$ is normal, $H_0$ is not, and we would not have been able to perform the same argument as above when restricting to $H_0\cap gH_0g^{-1}$, which only contains matrices with `central component' contained in the kernel of $\mu$. 

In order to establish admissibility, it is convenient to introduce further notation. Let $\epsilon=1$ for $p>2$ and $\epsilon=2$ if $p=2$, and let
\begin{equation*}
	U=\left\{\begin{pmatrix}
		1&p^\epsilon\mathbb{Z}_p&p^\epsilon\mathbb{Z}_p\\
		0&1&p^\epsilon\mathbb{Z}_p\\
		0&0&1
	\end{pmatrix}\right\}\subseteq H_0,
\end{equation*}
an open normal uniform pro-$p$ subgroup of $H_0$, with topological basis 
\begin{equation*}
	h_1=\begin{pmatrix}
		1&p^\epsilon&0\\
		0&1&0\\
		0&0&1
	\end{pmatrix}, \ h_2=\begin{pmatrix}
	1&0&0\\
	0&1&p^\epsilon\\
	0&0&1
	\end{pmatrix}, \ z=\begin{pmatrix}
	1&0&p^\epsilon\\
	0&1&0\\
	0&0&1
	\end{pmatrix}.
\end{equation*}

Since $H$ is normal, $gK_{\chi, \mu}$ is a character for $U$, where any $h\in U$ acts as $\widetilde{\chi}(g^{-1}hg)$. 

Now let $r_n=p^{-1/p^n}$ and consider $D(U)\cong \varprojlim D_{r_n}(U)$ as before. Recall that $D_{r_n}(U)$ is endowed with a Banach norm $|-|_{r_n}$ such that $|h_i-1|_{r_n}=r_n$ for $i=1, 2$, $|z-1|_{r_n}=r_n$. By construction, we then have
\begin{equation*}
	D_{r_n}(U)\otimes_{D(U)} gK_{\chi, \mu}=0
\end{equation*}
if $|1-\widetilde{\chi}(g^{-1}h_ig)|>r_n$ for $i=1$ or $2$. Indeed, in this case
\begin{equation*}
	\left|\frac{h_i-1}{\widetilde{\chi}(g^{-1}h_ig)-1}\right|_{r_n}<1,
\end{equation*}
so that $1-\frac{h_i-1}{\widetilde{\chi}(g^{-1}h_ig)-1}$ is a unit in the Banach algebra $D_{r_n}(U)$, but at the same time it acts as zero on $gK_{\chi, \mu}$, forcing
\begin{equation*}
	D_{r_n}(U)\otimes_{D(U)}gK_{\chi, \mu}=0,
\end{equation*}
as claimed. This observation will help us to invoke Corollary \ref{coadcheck} when investigating admissibility.

\begin{prop}
	The locally analytic $G$-representation $V_{\chi, \mu}$ is admissible. 
\end{prop}
\begin{proof}
	By Corollary \ref{coadcheck}, it suffices to show that for each $n$, there exist only finitely many double cosets $HgH=gH$ such that
	\begin{equation*}
		D_{r_n}(U)\otimes_{D(U)} gK_{\chi, \mu}\neq 0.
	\end{equation*}
	
	Coset representatives are given by 
	\begin{equation*}
		g_{a, c}:=\begin{pmatrix}
			1& a& 0\\
			0&1&c\\
			0&0&1
		\end{pmatrix}
	\end{equation*}
	for $a, c$ coset representatives of $\mathbb{Q}_p/\mathbb{Z}_p$. Fix $n$ such that
	\begin{equation*}
		|\chi(p^\epsilon, 0)-1|, |\chi(0, p^\epsilon)-1|<p^{-1/p^n(p-1)},
	\end{equation*}
	
	and note that $h_1$ acts on $g_{a, c}K_{\chi, \mu}$ as
	\begin{equation*}
		\widetilde{\chi}(g_{a, c}^{-1}h_1g_{a, c})=\widetilde{\chi}(\begin{pmatrix}
			1&p^\epsilon&p^\epsilon c\\
			0&1&0\\
			0&0&1
		\end{pmatrix})=\chi(p^\epsilon, 0)\cdot \mu(p^\epsilon c),
	\end{equation*} 
	while $h_2$ acts as
	\begin{equation*}
		\widetilde{\chi}(g_{a, c}^{-1}h_2g_{a, c})=\widetilde{\chi}(\begin{pmatrix}
			1&0&-p^\epsilon a\\
			0&1&p^\epsilon\\
			0&0&1
		\end{pmatrix})=\chi(0, p^\epsilon)\cdot \mu(-p^\epsilon a).
	\end{equation*}
	
	Since $\mathrm{ker}\mu=\mathbb{Z}_p$, if $x\in \mathbb{Q}_p$ with $|x|=p^{-i}$, then $\mu(x)$ is a primitive $p^i$th root of unity, and
	\begin{equation*}
		|\mu(x)-1|=p^{-1/p^{i-1}(p-1)}
	\end{equation*}
	by \cite[Example 2.1.6]{Kedlaya}.
	
	Therefore, if $c\in \mathbb{Q}_p$ such that $|c|>p^{-n}$, then
	\begin{equation*}
		|\mu(p^\epsilon c)-1|>p^{-\frac{1}{p^{n+\epsilon-1}(p-1)}}>p^{-1/p^n},
	\end{equation*}
	and hence
	\begin{equation*}
		|\chi(p^\epsilon,0)\cdot \mu(p^\epsilon c)-1|=|\chi(p^\epsilon,0)(\mu(p^\epsilon c)-1)+(\chi(p^\epsilon,0)-1)|=|\mu(p^\epsilon c)-1|>r_n.
	\end{equation*}
	
	By the same argument, if $|a|>p^{-n}$, then
	\begin{equation*}
		|\chi(0, p^\epsilon)\cdot \mu(-p^\epsilon a)-1|>r_n.
	\end{equation*}
	
	This implies that for each $n$, we have
	\begin{equation*}
		D_{r_n}(U)\otimes_{D(U)}gK_{\chi, \mu}=0
	\end{equation*}
	for all but finitely many double cosets $HgH$. Therefore, we can apply Corollary \ref{coadcheck} to deduce that $V_{\chi, \mu}$ is admissible.
\end{proof}

\subsection{The Borel subgroup of $\mathrm{SL}_2(\mathbb{Q}_p)$}

We now turn to a second example. For simplicity, let $p>2$ and let
\begin{equation*}
	G=\left\{\begin{pmatrix}
		a&b\\
		0& a^{-1}
	\end{pmatrix}| a\in \mathbb{Q}_p^\times, b\in \mathbb{Q}_p\right\}
\end{equation*}	
be the standard Borel subgroup of $\mathrm{SL}_2(\mathbb{Q}_p)$ consisting of upper-triangular matrices, and
\begin{equation*}
	H_0=\left\{\begin{pmatrix}
		a&b\\
		0&a^{-1}
	\end{pmatrix}\in G| a\in 1+p\mathbb{Z}_p, \ b\in p\mathbb{Z}_p\right\}
\end{equation*}
the matrices in $G$ which are congruent to the identiy matrix $I_2$ modulo $p$. Note that as $p>2$, $H$ is an open uniform pro-$p$-subgroup, with a topological basis given by
\begin{equation*}
	h_1=\begin{pmatrix}
		\gamma&0\\
	0&\gamma^{-1}
	\end{pmatrix}, \ h_2=\begin{pmatrix}
	1&p\\
	0&1
	\end{pmatrix}
\end{equation*}
for $\gamma=\mathrm{exp}(p)$.

The centre of $G$ consists of $\pm I_2$.

We first define a finite-dimensional irreducible $H=H_0Z_G$-representation as follows: Let $\sigma: (p\mathbb{Z}_p, +)\to K^\times$ be a non-trivial smooth character with $\mathrm{ker}\sigma=p^{j+1}\mathbb{Z}_p$, $j\geq 1$, and let $\chi: (\mathbb{Z}_p, +)\to K^\times$ be any locally analytic character.

Since $h_1^{p^{j-1}}h_2h_1^{-p^{j-1}}h_2^{-1}\in \langle h_2^{p^j}\rangle=\left(\begin{smallmatrix}
	1&p^{j+1}\mathbb{Z}_p\\0&1
\end{smallmatrix}\right)$, we obtain a $1$-dimensional representation $W_{\chi, \sigma}$ of $H_{j-1}:=\langle h_1^{p^{j-1}}, h_2\rangle$, given by
\begin{equation*}
	h_1^{p^{j-1}}\mapsto \chi(1), \ h_2\mapsto \sigma(p).
\end{equation*}

We set $V_{\chi, \sigma}=\mathrm{Ind}_{H_{j-1}}^{H_0} W_{\chi, \sigma}$. Explicitly, $V_{\chi, \sigma}$ is a $p^{j-1}$-dimensional $K$-vector space (with basis vectors $e_0, \hdots, e_{p^{j-1}-1}$ in natural bijection with $h_1^i$, $i=0, \hdots, p^{j-1}-1$), such that the action of $h_1$ is given by the matrix
\begin{equation*}
	\begin{pmatrix}
		0&0&\hdots&0&\chi(1)\\
		1&0&\ddots&0&0\\
		0&1&\ddots&\ddots&\vdots\\
		\vdots&\ddots&\ddots&\ddots&0\\
		0&0&\hdots&1&0
	\end{pmatrix}
\end{equation*}
and the action of $h_2$ is given by a diagonal matrix with entries $\sigma(p)$, $\sigma(\gamma^2 p)$, $\sigma(\gamma^4 p)$, $\hdots$, $\sigma(\gamma^{2(p^{j-1}-1)} p)$. This is an irreducible representation: Since $\mathrm{ker}(\sigma)=p^{j+1}\mathbb{Z}_p$, $h_2$ acts diagonally with distinct eigenvalues (note that $|1-\gamma^{2i}|=|ip|$, so for $i=1, \hdots, p^{j-1}-1$, we have $1-\gamma^2i\notin p^j\mathbb{Z}_p$, and thus $\gamma^{2i}-\gamma^{2i'}\notin p^j\mathbb{Z}_p$ for $i\neq i'\leq p^{j-1}-1$), so any non-zero subrepresentation contains at least one of the basis vectors, and hence all of them using the $h_1$-action.

As $Z_G=\pm I_2$, we can regard $V_{\chi, \sigma}$ as an irreducible, finite-dimensional (in particular, admissible), locally analytic $H=H_0Z_G$-representation by letting $-I_2$ act trivially. It is easy to check that the arguments below can also be adapted to the choice that $-I_2$ acts as $-1$.

We now consider the $G$-representation
\begin{equation*}
	X_{\chi, \sigma}:=\cind_H^G V_{\chi, \sigma}.
\end{equation*} 

\begin{prop}
	The locally analytic $G$-representation $X_{\chi, \sigma}$ is topologically irreducible.
\end{prop}

\begin{proof}
	We verify the condition from Theorem \ref{FSMackey}. By the same argument as above $M=V'_{\chi, \sigma}$ is uniformly simple. 
	
	Let
	\begin{equation*}
		g=\begin{pmatrix}
			ap^i& b\\
			0&a^{-1}p^{-i} 
		\end{pmatrix}\in G\setminus H
	\end{equation*} 
	with $a\in \mathbb{Z}_p^\times$, $i\in \mathbb{Z}$, and $b\in \mathbb{Q}_p$.
	
	If $i>0$, let $c=\mathrm{max}\{1, p^{j-2i}\}$. Then $gh_2^{c}g^{-1}=h_2^{ca^2p^{2i}}\in H_0\cap gH_0g^{-1}$ acts trivially on $V_{\chi, \sigma}$, whereas $h_2^c$ does not.
	
	Similarly, if $i<0$, let $c=p^{j-2i-1}$. Then $gh_2^{c}g^{-1}=h_2^{a^2p^{j-1}}\in H_0\cap gH_0g^{-1}$ acts non-trivially on $V_{\chi, \sigma}$, whereas $h_2^c$ is in the kernel.
	
	It thus remains to consider the case $i=0$. Note that then $gh_2g^{-1}=h_2^{a^2}\in H_0\cap gH_0g^{-1}$. The eigenvalues of $h_2$ on $V_{\chi, \sigma}$ where given by $\sigma(\gamma^{2i} p)$ for $i=0, 1, \hdots, p^{j-1}-1$. If $a^2\notin 1+p\mathbb{Z}_p$, then this set of eigenvalues is disjoint from $\sigma(a^2\gamma^{2i'} p)$ for $i'=0, 1, \hdots, p^{j-1}-1$, as $a^2\mathrm{exp}(2i'p)$ and $\mathrm{exp}(2i p)$ already differ modulo $p$.
	
	We have now reduced to the case $i=0$ and (without loss of generality) $a\in \{1, -1\}$. Using multiplication with a central element, we can suppose that $a=1$.
	
	If $|b|\geq 1$, then $|\frac{p}{p+ b}|\leq |p|$, so that there exists $x\in 1+p\mathbb{Z}_p$ with $x^2=1-\frac{p}{p+b}=\frac{b}{p+b}$. In particular, $px^2=b-bx^2$, so that $b(x^{-1}-x)=px$. Set $h=\left(\begin{smallmatrix}
		x&0\\0&x^{-1}
	\end{smallmatrix}\right)$. Then
	\begin{equation*}
		ghg^{-1}=\begin{pmatrix}
		x & b(x^{-1}-x)\\
		0&x^{-1}
		\end{pmatrix}=h\cdot h_2\in H_0\cap gH_0g^{-1}.
	\end{equation*} 
	
	Let $\alpha\in \mathbb{Z}_p$ such that $x=\mathrm{exp}(\alpha p)$, so that $h=h_1^\alpha$. After extending $K$, we can assume that $K$ contains all $p^{j-1}$th roots of $\chi(1)$, so that $h_1$, and hence $h$, acts semisimply on $V_{\chi, \sigma}$. Each eigenvalue of $h$ then has the property that its $p^{j-1}$th power equals $\chi(1)^\alpha$.

	We can give a similar description of the action of $hh_2$. In fact, if $k\in \{0, \hdots, p^{j-1}-1\}$ such that $\alpha-k\in p^{j-1}$, let $l=p^{1-j}(\alpha-k)\in \mathbb{Z}_p$. Then the action of $hh_2$ is given by sending the basis vector $e_i$ to
	\begin{equation*}
		hh_2\cdot e_i=\begin{cases}
			e_{i+k}\chi(1)^l \sigma(\gamma^{2i}p)\ \text{ if $i+k<p^{j-1}$}\\
			e_{i+k-p^{j-1}}\chi(1)^{l+1}\sigma(\gamma^{2i}p) \ \text{if $i+k\geq p^{j-1}$}.
		\end{cases}
	\end{equation*}

	Note $(hh_2)^{p^{j-1}}$ acts on $V_{\chi, \sigma}$ as a scalar multiple of the identity, with each eigenvalue being
	\begin{align*}
		\chi(1)^\alpha \cdot \prod_{i=0}^{p^{j-1}-1} \sigma(\gamma^{2i}p)&=\chi(1)^\alpha \cdot \sigma(\sum_i \gamma^{2i}p)\\
		&=\chi(1)^\alpha\cdot \sigma(p\cdot \frac{\gamma^{2p^{j-1}}-1}{\gamma^2-1}).
	\end{align*}
	
	After extending $K$ to contain suitable $p^{j-1}$th roots, we can thus assume that $hh_2$ acts semisimply on $V_{\chi, \sigma}$, and each eigenvalue is a $p^{j-1}$th root of the expression above.
	
	We claim that none of the eigenvalues of $h$ are eigenvalues of $hh_2$. Indeed, as $|\gamma^{2p^{j-1}}-1|=|p|^j$, one has $\sigma(p\cdot \frac{\gamma^{2p^{j-1}}-1}{\gamma^2-1})\neq 1$. Therefore, there can be no $H_0\cap gH_0g^{-1}$-linear map between $gV_{\chi, \sigma}$ and $V_{\chi, \sigma}$ in this case.
	
	We have thus shown that $\mathrm{Hom}_{H_0\cap gH_0g^{-1}}(V_{\chi, \sigma}, gV_{\chi, \sigma})=0$ is zero, unless $g\in H$. Note that for $g\in H$, the endomorphism space $\mathrm{Hom}_{H_0}(V_{\chi, \sigma}, V_{\chi, \sigma})$ is indeed $K$, as $h_2$ acts diagonally with $p^{j-1}$ distinct eigenvalues (so any endomorphism needs to be given by a diagonal matrix) and the only diagonal matrices commuting with the $h_1$-action are scalar matrices.
	
	We have thus verified that the conditions of Theorem \ref{FSMackey} are satisfied, and $X_{\chi, \sigma}$ is a topologically irreducible $G$-representation.
\end{proof}

\begin{prop}
	The representation $X_{\chi, \sigma}$ is not admissible.
\end{prop}
\begin{proof}
	We will once again use Corollary \ref{coadcheck}. In fact, using the Mackey decomposition of $X_{\chi, \sigma}$ into $H_0$-representations, we will exhibit infinitely many pairwise isomorphic constituents. 
	
	Recall that $j$ is chosen such that $\mathrm{ker}(\sigma)=p^{j+1}\mathbb{Z}_p$, so that $h_2^{p^j}$ acts trivially on $V_{\chi, \sigma}$.
	
	For each $i\geq j$, let
	\begin{equation*}
		g_i:=\begin{pmatrix}
			p^{-i}& p^{-i}\\
			0& p^i
		\end{pmatrix}\in G.
	\end{equation*}
	Note that $g_i, g_{i'}$ are in different $H$-double cosets for $i\neq  i'$, as the respective first diagonal entries have different absolute values.
	
	Then
	\begin{equation*}
		h_1=g_i\begin{pmatrix}
			\gamma & \gamma-\gamma^{-1}\\
			0& \gamma^{-1}
		\end{pmatrix}g_i^{-1}\in H_0\cap g_iH_0g_i^{-1}
	\end{equation*}
	and
	\begin{equation*}
		h_2=g_i h_2^{p^i}g_i^{-1}\in H_0\cap g_iH_0g_i^{-1},
	\end{equation*}
	so that $H_0\cap g_iH_0g_i^{-1}=H_0$.
	
	Writing $M:=(V_{\chi, \sigma})'$, we thus obtain $D(H_0)\otimes_{D(H_0\cap g_iH_0g_i^{-1})}g_iM=g_iM$. Note that in this example, $H\cap g_iHg_i^{-1}=(H_0\cap g_iH_0g_i^{-1})Z_G$ (if $h\in H$ with $g_ihg_i^{-1}\in H_0$, then the first entry of $h$ needs to be $1$ mod $p$, and as $p>2$, this forces $h\in H_0$), so we can conclude that
	\begin{equation*}
		D(H)\otimes_{D(H\cap g_iHg_i^{-1})}g_iM=g_iM.
	\end{equation*}
	
	Since $i\geq j$, $h_2$ acts trivially on $g_iM$ (because $h_2^{p^i}$ acts trivially on $M$), and $h_1$ acts on $g_iM$ in the same way as $\left(\begin{smallmatrix}
		\gamma& \gamma-\gamma^{-1}\\0&\gamma^{-1}
	\end{smallmatrix}\right)$ acts on $M$. This is independent of $i$, so for all $i\geq j$, the $D(H_0)$-modules $g_iM$ are pairwise isomorphic.
	
	This might seem somewhat counterintuitive if we compare this to the irreducibility criterion above, but let us emphasize that this yields an isomorphism of $g_iV_{\chi, \sigma}$ and $g_{i'}V_{\chi, \sigma}$ as representations of $H_0=H_0\cap g_iH_0g_i^{-1}\cap g_{i'}H_0g_{i'}^{-1}$, whereas the earlier proposition can be phrased as saying that there is no morphism between these for $i\neq i'$ as $g_iH_0g_i^{-1}\cap g_{i'}H_0g_{i'}^{-1}$-representations.
	
	Choosing $n$ large enough such that
	\begin{equation*}
		D_{r_n}(H_0)\otimes_{D(H_0)} g_jM\neq 0,
	\end{equation*}
	the result follows from Corollary \ref{coadcheck}.
\end{proof}

\end{document}